\documentclass[a4paper,10pt]{amsart}
\usepackage{mathrsfs}
\usepackage{cases}

\usepackage{amsfonts}

\usepackage{graphicx}
\usepackage{amsmath}
\usepackage{amssymb}
\usepackage[all]{xypic}
\usepackage[all]{xy}
\usepackage{color}
\usepackage{colordvi}
\usepackage{multicol}

\begin{document}
\input xy
\xyoption{all}

\renewcommand{\mod}{\operatorname{mod}\nolimits}
\newcommand{\proj}{\operatorname{proj.}\nolimits}
\newcommand{\rad}{\operatorname{rad}\nolimits}
\newcommand{\soc}{\operatorname{soc}\nolimits}
\newcommand{\ind}{\operatorname{inj.dim}\nolimits}
\newcommand{\ch}{\operatorname{ch}\nolimits}
\renewcommand{\H}{\operatorname{H}\nolimits}
\newcommand{\id}{\operatorname{id}\nolimits}
\newcommand{\Mod}{\operatorname{Mod}\nolimits}
\newcommand{\End}{\operatorname{End}\nolimits}
\newcommand{\Ob}{\operatorname{Ob}\nolimits}
\newcommand{\Ht}{\operatorname{Ht}\nolimits}
\newcommand{\cone}{\operatorname{cone}\nolimits}
\newcommand{\rep}{\operatorname{rep}\nolimits}
\newcommand{\Ext}{\operatorname{Ext}\nolimits}
\newcommand{\Hom}{\operatorname{Hom}\nolimits}
\newcommand{\Div}{\operatorname{Div}\nolimits}
\newcommand{\RHom}{\operatorname{RHom}\nolimits}
\renewcommand{\deg}{\operatorname{deg}\nolimits}
\renewcommand{\Im}{\operatorname{Im}\nolimits}
\newcommand{\Ker}{\operatorname{Ker}\nolimits}
\newcommand{\Coh}{\operatorname{Coh}\nolimits}
\newcommand{\CH}{\operatorname{CH}\nolimits}
\renewcommand{\dim}{\operatorname{dim}\nolimits}
\renewcommand{\div}{\operatorname{div}\nolimits}
\newcommand{\Ab}{{\operatorname{Ab}\nolimits}}
\renewcommand{\Vec}{{\operatorname{Vec}\nolimits}}
\newcommand{\pd}{\operatorname{proj.dim}\nolimits}
\newcommand{\sdim}{\operatorname{sdim}\nolimits}
\newcommand{\add}{\operatorname{add}\nolimits}
\newcommand{\pr}{\operatorname{pr}\nolimits}
\newcommand{\cc}{{\mathcal C}}
\newcommand{\ce}{{\mathcal E}}
\newcommand{\cf}{{\mathcal F}}
\newcommand{\cx}{{\mathcal X}}
\newcommand{\ct}{{\mathcal T}}
\newcommand{\cu}{{\mathcal U}}
\newcommand{\cv}{{\mathcal V}}
\newcommand{\ca}{{\mathcal A}}
\newcommand{\cb}{{\mathcal B}}
\newcommand{\cg}{{\mathcal G}}
\newcommand{\cw}{{\mathcal W}}
\newcommand{\co}{{\mathcal O}}
\newcommand{\cd}{{\mathcal D}}
\newcommand{\calr}{{\mathcal R}}
\newcommand{\ol}{\overline}
\newcommand{\ul}{\underline}
\newcommand{\st}{[1]}
\newcommand{\ow}{\widetilde}
\renewcommand{\P}{\mathbf{P}}
\newcommand{\pic}{\operatorname{Pic}\nolimits}
\newcommand{\Spec}{\operatorname{Spec}\nolimits}
\newtheorem{theorem}{Theorem}[section]
\newtheorem{acknowledgement}[theorem]{Acknowledgement}
\newtheorem{algorithm}[theorem]{Algorithm}
\newtheorem{axiom}[theorem]{Axiom}
\newtheorem{case}[theorem]{Case}
\newtheorem{claim}[theorem]{Claim}
\newtheorem{conclusion}[theorem]{Conclusion}
\newtheorem{condition}[theorem]{Condition}
\newtheorem{conjecture}[theorem]{Conjecture}
\newtheorem{construction}[theorem]{Construction}
\newtheorem{corollary}[theorem]{Corollary}
\newtheorem{criterion}[theorem]{Criterion}
\newtheorem{definition}[theorem]{Definition}
\newtheorem{example}[theorem]{Example}
\newtheorem{exercise}[theorem]{Exercise}
\newtheorem{lemma}[theorem]{Lemma}
\newtheorem{notation}[theorem]{Notation}
\newtheorem{problem}[theorem]{Problem}
\newtheorem{proposition}[theorem]{Proposition}
\newtheorem{remark}[theorem]{Remark}
\newtheorem{solution}[theorem]{Solution}
\newtheorem{summary}[theorem]{Summary}
\newtheorem*{thm}{Theorem}

\def \Z{{\Bbb Z}}
\def \F{{\Bbb F}}
\def \C{{\Bbb C}}
\def \N{{\Bbb N}}
\def \R{{\Bbb R}}
\def \Q{{\Bbb Q}}
\def \P{{\Bbb P}}
\def \K{{\Bbb K}}
\def \E{{\Bbb E}}

\title{Coxeter transformations of the derived categories of coherent sheaves}
\author[Chen]{Xinhong Chen}
\address{Department of Mathematics, Southwest Jiaotong University, Chengdu 610031, P.R.China}
\email{chenxh2007@aliyun.cn}

\author[Lu]{Ming Lu$^\dag$}
\address{Department of Mathematics, Sichuan University, Chengdu 610064, P.R.China}
\email{luming2012@aliyun.cn}
\thanks{$^\dag$ Corresponding author}


\begin{abstract}
In this paper, we study the Coxeter transformations of the derived categories of coherent sheaves on smooth complete varieties. We first obtain that if the rank  of the Grothendieck group is finite, say $m$, then its characteristic polynomials is $(\lambda+(-1)^{n})^m$, where $n$ is dimension of the variety. We then study the Jordan canonical forms of the Coxeter transformations for rational surfaces, smooth complete toric varieties with ample canonical or anticanonical bundles, and prove that their Jordan canonical forms can determine and be determined by the Betti numbers of these varieties.  As an application, we compute the Jordan canonical forms of tensor products of matrices.
\end{abstract}

\maketitle

\section{Introduction}
Coxeter transformations play an important role in the development of Lie group and Lie algebra theory. For a compact Lie group, a Coxeter transformation or a Coxeter element is defined to be the product of all the reflections
in the root system initially. H. S. Coxeter studies these elements and their eigenvalues in \cite{Coxeter51}. He also observes that the
number of roots in the corresponding root system is equal to the product of the order of the Coxeter transformation (called the Coxeter number) and the number of eigenvalues of the
Coxeter transformation, which is proved by A. J. Coleman in \cite{Col58}.  \\

In \cite{Ste85, SuSt78}, V. F. Subbotin and R. B. Stekolshchik consider the eigenvalues and the Jordan canonical forms of the Coxeter transformations for symmetrizable Cartan matrices. They prove that the Jordan canonical form of the Coxeter transformation is diagonal if and only if the Tits form is non-degenerate. For an extended Dynkin diagram, the Jordan canonical form of the Coxeter transformation contains only one $2\times2$ Jordan block and all eigenvalues have norm $1$. They also prove the following: if the Tits form is indefinite and degenerate, then the number of two-dimensional Jordan blocks
coincides with dimension of the kernel of the Tits form, and the remaining Jordan blocks are $1\times1$. C. M. Ringel generalizes their results to generalized Cartan matrices in \cite{Rin94}, he shows that the spectral radius $\rho$ of the Coxeter transformation has norm greater than $1$, and $\rho$ is also an eigenvalue with multiplicity one. A. J. Coleman computes characteristic polynomials for the Coxeter transformations of all extended Dynkin diagrams, including the case with cycles $\tilde{A}_n$ in \cite{Col89}. There is a nice survey \cite{Ste05} on this subject and recent development \\

The Coxeter transformation is also important in the study of representations of algebras,
quivers, partially ordered sets (posets) and lattices, see \cite{DR76, DR81, PT90, Pe94, Rin94}. Its distinguished role in these areas is
related to the construction of the Coxeter functors $DTr$ (see \cite{ARS95}) and given by I. N. Bernstein, I. M. Gelfand and
V. A. Ponomarev in \cite{BGP73}.
Given any path algebra $A$ with finite global dimension, let $C_A$ be the Cartan matrix. It is well-known that the Coxeter transformation coincides with $-C_A^tC_A^{-1}$. One interesting result is that the spectral radius $\rho$ of the Coxeter transformation is related to the representation type of the hereditary algebra. Concretely, let $A$ be a hereditary algebra.
Then

\noindent $\bullet$ $A$ is representation-finite if $\rho_A=1 $ and $1$ is not a root of the Coxeter polynomial $\chi_A$.

\noindent$\bullet$ $A$ is tame if $\rho_A= 1 \in \mathrm{Roots}(\chi_A)$.

\noindent$\bullet$ $A$ is wild if $ \rho_A>1$.
\\

When we consider the bounded derived category of a finite dimensional algebra $A$ and its Grothendieck group, the Coxeter transformation is just the linear transformation on the Grothendieck group induced by the functor $S_A[-1]$ (see, e.g., \cite{LP08}), where $S_A$ is the Serre functor (see \cite{BK90}) and $[-1]$ is the inverse of the shift functor which is denoted by $[1]$. So in general, we can define the Coxeter transformation of a triangulated category with Serre functor $S$ to be the linear transformation induced by the functor $S[-1]$ on Grothendieck group. It is an invariant under triangulated equivalences. About the derived equivalences of rings and algebras, J. Rickard get interesting results in \cite{Rick89, Rick91}.\\

Besides the bounded derived categories of finite dimensional algebras (see, e.g., \cite{Happel88}), another important class of triangulated categories is the bounded derived categories of coherent sheaves on (smooth projective) varieties, see, e.g., \cite{Bondal-Orlov97, Bri06}.
An interesting question is when the derived category of a variety is equivalent to the one of a finite dimensional algebra. In 1970's, A. A. Beilinson constructs an algebra derived equivalent to $\P^n$, in this case, tilting sheaves play an important role. In \cite{Bondal90, BK90}, A. I. Bondal and M. M. Kapranov give an useful method to find a tilting sheaf: (strongly) exceptional sequences of sheaves (in particular line bundles). After that, many smooth complete varieties are found to admit a strongly exceptional sequences of line bundles, and most of them are smooth complete toric varieties, \cite{HP08, Kin97, Per03, Per09}.
A. King conjectures that all smooth toric varieties admit a strongly exceptional sequence of line bundles in \cite{Kin97}, but this is proved false by L. Hille and M. Perling in \cite{HP06}.
\\

By Serre duality \cite{Hart77}, we know that the derived categories of smooth complete varieties admit Serre functors,
so we can define the Coxeter transformations for them.

The main goal of this paper is to investigate the Coxeter transformation of the derived category of coherent sheaves on smooth complete varieties. We compute their characteristic polynomials and Jordan canonical forms. We also find relations between these invariants and the geometric properties of the varieties. These invariants give necessary conditions for the bounded derived category of a smooth variety $X$ to be triangulated equivalent to the bounded derived category of a finite dimensional algebra $A$.\\

The paper is organized as follows. In Section 2, we recall some basic properties of smooth varieties and facts needed for our main results.
In Section 3, we prove that if the rank of the Grothendieck group of a smooth complete variety of dimension $n$ is finite, say $m$, then the characteristic polynomial of the Coxeter transformation is $(\lambda+(-1)^{n})^m$, see Theorem \ref{main theorem 1}. \\

In Section 4, we study the Jordan canonical forms of Coxeter transformations. We first prove that for a smooth complete toric variety with ample canonical or anticanonical
sheaf, its Betti numbers are uniquely determined by and also determine the number of the Jordan blocks in the Jordan canonical form of the Coxeter transformation and their dimensions, see Theorem \ref{main theorem 3}. We then prove that for a rational surface, the number of the Jordan blocks is equal to the rank of the Picard group; furthermore if the rank of the Picard group is not $10$, then the Jordan canonical form of the Coxeter transformation contains only one Jordan block of dimension 3, and the others are of dimension one; otherwise, it has two Jordan block of dimension $2$, and the others are of dimension $1$, see Theorem \ref{main theorem 2}. As a corollary, if a rational surface or a smooth complete toric variety $X$ with ample canonical or anticanonical divisor is derived equivalent to a finite dimensional algebra $A$, then the Betti numbers of $X$ are determined by the matrix $-C_A^tC_A^{-1}$, see Corollary \ref{corollary derived equivalent}.\\

In the final Section 5, we first prove that for two smooth toric varieties $X$ and $Y$, if the Jordan
canonical form of the Coxeter transformations of $X$ and $Y$ are
$\bigoplus_{i=1}^{t_1}J^X_i$ and $\bigoplus_{j=1}^{t_2}J^Y_j$ respectively, then the sizes and number of Jordan blocks in
the Jordan canonical form of the Coxeter transformation of $X\times Y$ are the same as that of
$\bigoplus_{i,j=1}^{t_1,t_2}J^X_i\boxtimes J^Y_j$, see Proposition \ref{theorem Jordan block}. We then give a formula to compute the Jordan canonical form for tensor product of invertible matrices, see Proposition \ref{proposition jordan block}. This formula generalizes a formula for the tensor product of two matrices in \cite{MV} and it is useful to construct graded local Frobenius algebras.

\begin{acknowledgement}
The authors deeply thank their supervisors Professor Bin Zhang and Professor Liangang Peng for helpful discussions.
\end{acknowledgement}

\section{Preliminaries}
Throughout this paper, we assume that the ground field $k$ is
an algebraically closed field of characteristic $0$.\\

Let $X$ be a variety over $k$. The category of coherent sheaves on
$X$ is denoted by $\Coh(X)$. The locally
free coherent sheaves on $X$ are called \emph{Vector bundles}, and they form a full subcategory of $\Coh(X)$ which is denoted by
$\Vec(X)$.
\begin{definition} 
The \emph{Grothendieck group of coherent sheaves} on a variety $X$,
denoted by $G_0(X)$, is defined to be generated by $[\cf]$, the
isomorphism classes of coherent sheaves on $X$, subject to the
relations $[\cf]=[\cf']+[\cf'']$ for every exact sequence
$$0\rightarrow\cf'\rightarrow\cf\rightarrow\cf''\rightarrow0.$$
\end{definition}

Similarly, we have the following definition.
\begin{definition} 
The \emph{Grothendieck group of vector bundles} on a variety $X$,
denoted by $K_0(X)$, is defined to be generated by $[\mathbb{E}]$,
the isomorphism classes of vector bundles on $X$, subject to the
relations $[\mathbb{E}]=[\mathbb{E}']+\mathbb{E}''$ for every exact
sequence
$$0\rightarrow\mathbb{E}'\rightarrow\mathbb{E}\rightarrow\mathbb{E}''\rightarrow0.$$
\end{definition}

In the affine case, every short exact sequence of vector bundles
splits. Therefore, $K_0(X)=K_0(A)$ for an affine variety $X=\Spec
A$, where $K_0(A)$ is the Grothendieck group of projective modules over $A$, see \cite{BG08}.

The embedding of the category of vector bundles on $X$ into the category of
coherent sheaves induces the \emph{Cartan homomorphism}
$$K_0(X)\rightarrow G_0(X).$$ This homomorphism is in general
neither injective nor surjective, but for a smooth variety $X$, it
is an isomorphism.
\begin{theorem}\label{theorem cartan hom} (see, e.g., \cite{BG08})
For a smooth variety, the Cartan homomorphism is an isomorphism.
\end{theorem}

Therefore, for a smooth variety $X$, we do not distinguish between $G_0(X)$ and $K_0(X)$, and just use $K_0(X)$ to denote it.\\

We denote the bounded derived category of coherent sheaves on $X$ by
$\cd^b(\Coh(X))$ or just $\cd^b(X)$. It is well known that
$\cd^b(X)$ is a triangulated category with shift functor
$[1]$. The Grothendieck group of $\cd^b(X)$ is defined to be
generated by the isomorphism class of complexes of coherent sheaves,
subject to the distinguished triangles.

The following lemma is well-known.
\begin{lemma}
Let $X$ be a variety over $k$. Then the Grothendieck group of
$\cd^b(X)$ is isomorphic to $G_0(X)$. In particular, if $X$ is
smooth, then the Grothendieck group of $\cd^b(X)$ is isomorphic to
$K_0(X)$.
\end{lemma}
In fact, for any abelian category $\ca$, we have that the Grothendieck group of $\cd^b(\ca)$ is isomorphic to $K_0(\ca)$. For this reason, we do not distinguish the Grothendieck group of $\cd^b(\ca)$ and $K_0(\ca)$ in the following. \\

The tensor product of vector bundles induces a commutative ring
structure on $K_0(X)$, we call it the \emph{Grothendieck ring} of $X$. In general, $G_0(X)$ is a $K_0(X)$-module.\\

\begin{remark} (\cite{Hart77})
Even though $X$ is a smooth complete variety of dimension $n$, we know that the Grothendieck group $K_0(X)$ is possibly huge, i.e., with infinite rank.
\end{remark}

In the following, we recall basic facts of \emph{Chow groups} of varieties, see \cite{Ful98}. Let
$Z_k(X)$ denote the group of $k$-dimensional algebraic cycles on a variety $X$.

The corresponding quotient group of $Z_k(X)$ modulo the rational equivalences is called the
\emph{Chow group} of $k$-cycles, denoted by $A_k(X)$. Let
$$A_\ast(X)=\bigoplus_{i=1}^nA_k(X),\quad n=\dim X.$$

For a general (not necessarily smooth) complete $n$-dimensional
variety $X$, the group of cycles of codimension $k$ modulo rational
equivalence is denoted by $\CH^k(X)=A_{n-k}(X)$. If in addition $X$ is smooth, then there is a
product
$$
\CH^k(X)\times \CH^l(X)\rightarrow \CH^{k+l}(X),\quad (V_1,V_2)\mapsto \delta^*(V_1\times V_2)
$$
where $\delta: X\rightarrow X\times X$ is the diagonal embedding. It coincides
with the geometric intersection of cycles in the case of transverse
intersections. We denote this product by $\alpha\cdot \beta$ for
$\alpha\in \CH^k(X)$ and $\beta\in \CH^l(X)$. This product
corresponds to cup product in cohomology, and there is a ring
homomorphism
$$
\CH^\ast(X)\rightarrow H^\ast(X,\Z).
$$
But this morphism is neither monomorphic nor epimorphic in general, \cite{BG08}.\\

Let $X$ be a variety of dimension $n$. The \emph{Chern character}
$\mathrm{ch}(\E)$ of a vector bundle $\E$ of rank $r$ is given by
$$
\mathrm{ch}(\E)=\sum_{k=0}^\infty
\frac{p_k(\E)}{k!}\in\CH^\ast(X)_\Q,
$$
where
$$
p_k(\E)=x_1(\E)^k+\cdots+ x_r(\E)^k, \quad x_i(\E)'s\mbox{ are the chern roots of }\E.
$$
In particular, for a line bundle $\mathbb{L}$ and its corresponding Cartier
divisor $D$ on $X$, we have
$$
\mathrm{ch}(\mathbb{L})=1+c_1(\mathbb{L})+\frac{c_1(\mathbb{L})^2}{2!}+\cdots+\frac{c_1(\mathbb{L})^n}{n!}
$$
$$
\quad\quad\quad\quad\,\,=1+D+\frac{D^2}{2!}+\cdots+\frac{D^n}{n!},\,n=\dim(X).
$$
\\

For an abelian group $B$, we use the notation $B_\Q:=B\otimes \Q$.
\begin{theorem} (\cite{BG08, Ful98})\label{lemma chern map}
Let $X$ be a smooth complete variety. There is a ring isomorphism
$$
(\mathrm{ch}_X)_\Q:K_0(X)_\Q\rightarrow \CH^\ast(X)_\Q, \quad [\E]\mapsto \mathrm{ch}_X(\E).
$$
\end{theorem}

For a normal variety $X$, the sheaf of Zariski
$p$-forms is defined to be
$$\hat{\Omega}^p_X=(\Omega^p_X)^{\vee\vee}=j_*\Omega^p_{U_0},$$
where $j:U_0\hookrightarrow X$ is the
inclusion of the smooth locus of $X$.
\begin{definition} (see, e.g., \cite{Hart77}) The \emph{canonical
sheaf} of a normal variety $X$ is $\omega_X=\hat{\Omega}^n_X$, where
$n$ is the dimension of $X$. This is a reflexive sheaf of rank 1, so
that $\omega_X\cong \co_X(D)$ for some Weil divisor $D$ on $X$.
This divisor is called the \emph{canonical divisor} of $X$, denoted by $K_X$.
\end{definition}
Notice that for a smooth variety $X$, the
canonical divisor is a Cartier divisor.

\begin{theorem}[Serre duality](see, e.g., \cite{Hart77})
\label{theorem serre duality} Let $\omega_X$ be the canonical sheaf
of a complete normal Cohen-Macaulay variety $X$ of dimension $n$.
Then for every coherent sheaf $\cf$ on $X$, there is a natural
isomorphism
$$D H^p(X,\cf)\simeq \Ext^{n-p}_{\co_X}(\cf,\omega_X),$$
$\mbox{ where }D=\Hom(-,k) \mbox{ is the duality functor.}$
\end{theorem}

So for any $\cf,\cg$ on $X$, there is a natural isomorphism
$D\Hom_{\co_X}(\cg,\cf)\simeq \Ext^n_{\co_X}(\cf,\cg\otimes
\omega_X)$. This motivates the following definition.

\begin{definition}[Serre functor] (\cite{BK90}) Let $\ca$ be a $k$-linear
Hom-finite category. An autoequivalent functor $S_A$ is called a
\emph{Serre functor} of $\ca$, if there exists a natural
isomorphism $$D\Hom_\ca(X,Y)\cong \Hom_\ca(Y,S_\ca X)$$ for any
object $X,Y\in\ca$, where $D=\Hom(-,k)$.
\end{definition}
In particular, a Serre functor is unique up to a natural isomorphism.\\

When we consider the bounded derived category $\cd^b(X)$ for a smooth complete variety $X$ of dimension $n$, by
Theorem \ref{theorem serre duality}, we know that $\cd^b(X)$ has a
Serre functor $-\otimes_{\co_X}\omega_X[n]$. The importance of the
existence of Serre functor in a Krull-Schmidt triangulated category is that
it is equivalent to the existence of Auslander-Reiten triangles in it. The
Auslander-Reiten functor is equivalent to $S_\ca[-1]$, so the
functor $-\otimes\omega_X[n-1]$ is the Auslander-Reiten functor of
$\cd^b(X)$.
In the case of finite dimensional algebras, the Coxeter transformations are defined to be induced by the Auslander-Reiten functors on Grothendieck groups.
So we have the following definition similar to that of finite dimensional algebras.

\begin{definition}
Let $X$ be a smooth complete variety of dimension $n$ and $\omega_X$ be its canonical sheaf. Then the 
transformation induced by $-\otimes\omega_X[n-1]$ on $K_0(X)_\Q$ is called
the \emph{Coxeter transformation} of $X$. Furthermore, if the rational Grothendieck group $K_0(X)_\Q$ is finite dimensional, then the characteristic polynomial of the Coxeter transformation is called the \emph{Coxeter polynomial} of $X$.
\end{definition}

If two triangulated categories with Serre functors are triangulated equivalent, then their
Grothendieck groups are isomorphic and equivalent functors commute
with Serre functors. It is easy to see that the Coxeter
polynomial is an invariant of triangulated equivalences.

\section{Coxeter polynomials}

In Section 2, we define the Coxeter polynomials for a smooth complete variety. In this Section, we compute it explicitly.

\begin{theorem}\label{main theorem 1}
Let $X$ be a smooth complete variety of dimension $n$ and the rank $m$
of the Grothendieck group $K_0(X)$ be finite. Then the Coxeter polynomial of $X$ is $(\lambda+(-1)^{n})^m$. In particular, all eigenvalues of the Coxeter
transformation are $(-1)^{n-1}$.
\end{theorem}
\begin{proof}
Because $X$ is a smooth complete variety, we have the following ring isomorphism
$$(\ch_X)_\Q:K_0(X)_\Q\rightarrow \CH^\ast (X)_\Q, \quad [\mathbb{E}]\mapsto \ch_X(\mathbb{E}).$$

The Coxeter polynomial is the characteristic
polynomial of the induced linear transformation $\Phi$ on
Grothendieck group by the Auslander-Reiten functor
$-\otimes\omega_X[n-1]$. For any object $W^{*}\in \cd^b(X)$, we have
$[W^{*}[n-1]]=(-1)^{n-1}[W^{*}]$, so we only need compute the induced linear
transformation $\Phi'$ on Grothendieck group by the functor
$-\otimes\omega_X$. The eigenvalues of $\Phi$ are the ones of
$\Phi'$ multiplied by $(-1)^{n-1}$.\\

Let $\Psi$ be the induced transformation on $\CH^\ast(X)_\Q$ by the intersection product
with $\ch_X(\omega_X)$. For any vector bundle $W$ on $X$, we have
$$\begin{array}{ccc}(\ch_X)_\Q\Phi'([W])&=(\ch_X)_\Q( [W\otimes\omega_X])\quad\quad\quad\quad\\
&=(\ch_X)_\Q([W])\cdot(\ch_X)_\Q([\omega_X])\\
&=\Psi(\ch_X)_\Q([W]).\quad\quad\quad\quad\quad\,\,\,\end{array}$$
Because $K_0(X)$ is generated by vector bundles,
we have the following commutative diagram. 
\[\xymatrix{ K_0(X)_\Q \ar[rr]^{(\ch_X)_\Q} \ar[dd]^{\Phi'}&& \CH^\ast(X)_\Q \ar[dd]^{\Psi}\\
\\
K_0(X)_\Q\ar[rr]^{(\ch_X)_\Q } && \CH^\ast(X)_\Q}\]
Since $(\ch_X)_\Q$ is an isomorphism,
it is easy to see that the matrix representing the linear
transformation $\Phi'$ is similar to the matrix representing $\Psi$, so we
only need compute the eigenvalues of $\Psi$.\\

For any element $[E_i]\in \CH^i(X)_\Q$, $[F_j]\in\CH^j(X)_\Q$, we have
$$[E_i]\cdot [F_j]\in \CH^{i+j}(X)_\Q \mbox{ and } \Psi([E_i])=[E_i\cdot
(\ch_X(\omega_X))].$$
Since $\omega_X=\co(K_X)$,
$$\ch_X(\co_X(K_X))=1+(K_X)+\frac{(K_X)^2}{2!}+\cdots+ \frac{(K_X)^n}{n!}\in
\CH^*(X,\Q).$$

Let $b_i$ be the rank of $\CH^i(X)$. We choose classes
$[E_{ij}]\in \CH^i(X)$ for $0\leq i\leq n$, $1\leq j\leq b_i$, such that $\{[E_{ij}]|j=1,\dots, b_i\}$ form a basis of
$\CH^i(X)_\Q$.

For any element $[E_{ij}]$,
$$\Psi([E_{ij}])=[E_{ij}\cdot
(\ch_X(\omega_X))]=[E_{ij}]+[E_{ij}\cdot K_X]+\cdots+
\frac{[E_{ij}\cdot K_X^{n-i}]}{(n-i)!}.$$ Using this basis, we can
write $\Psi([E_{ij}])$ as a linear combination of $$\{E_{kl}|0\leq k\leq n, 1\leq l\leq b_k\}.$$ For
$[E_{ij}\cdot (K_X)^k]\in \CH^{i+k}(X)$, it is a linear
combination of $[E_{(i+k)l}]$ since $$\{[E_{(i+k)l}]| 1\leq l\leq b_{i+k} \}$$ is a basis of
$\CH^{i+k}(X)$. So the matrix representing $\Psi$ under the basis
$[E_{ij}]$ is an $m\times m$ upper triangular matrix with $1$ in the diagonal. So all the eigenvalues of
$\Psi$ are $1$. Furthermore, we know that the Coxeter polynomial is $(\lambda+(-1)^{n})^m$. The proof is complete.
\end{proof}

The following definition is practically important to find an algebra derived equivalent to a smooth projective variety.
\begin{definition} (\cite{Beilinson79, BK90})
A coherent sheaf $T$ on a smooth projective variety $X$ is called a
\emph{tilting sheaf} if

(1) It has no higher self-extensions, i.e. $\Ext^i(T,T)=0$ for all
$i>0$,

(2) The endomorphism algebra of $T$, $A=\Hom_X(T,T)$, has finite
global homological dimension,

(3) the direct summands of $T$ generate the bounded derived category
$\cd^b(X)$.
\end{definition}

The importance of tilting sheaf is that it induces an equivalence
between $\cd^b(X)$ and $\cd^b(A)$ (the bounded derived category of
finite generated right $A$-modules). If $T$ is a vector bundle, we call it
a \emph{tilting bundle}. If $\cd^b(X)$ admits a tilting sheaf, then its
Grothendieck group is free.

For a finite dimensional algebra with finite global dimension, it is well-known that
$D^b(A)$ admits Auslander-Reiten triangles such that the Auslander-Reiten translation $\tau$ is an equivalence, see \cite{Happel88, Happel91}.
On the level of the Grothendieck group, $\tau$ induces the Coxeter transformation
$\Phi_A$.

Let $S_1,\dots, S_n$ be a complete system of pairwise non-isomorphic simple $A$-modules,
$P_1,\dots , P_n$ be the corresponding projective covers and $I_1, \dots , I_n$ be the injective
envelopes. The integers $c_{ij} = \dim_k \Hom(P_i, P_j)$ yield an integral $n\times n$ matrix
$C_A = (c_{ij})$ with the property that $\Phi = -C_A^tC_A^{-1}$ represents the Coxeter
transformation $\Phi_A$ under the basis of $K_0(A)$ formed by the classes $[P_1],\dots , [P_n]$.\\

From Theorem \ref{main theorem 1}, we have the following corollary directly.
\begin{corollary}
Let $A$ be a finite dimensional algebra over $k$ with finite global
dimension. If $\cd^b(A)$ is triangulated equivalent to a bounded derived category
of a smooth complete variety, then all the eigenvalues of the
Coxeter transformation of $A$ are equal, either $1$ or
$-1$.
\end{corollary}
\begin{proof} We only need the fact that the Grothendieck group of a finite dimensional algebra is free and its rank is finite.
\end{proof}
\begin{example}\label{example projective space}
Let $\vec{Q}$ be the following quiver:
\[\xymatrix{\circ^1\ar@<4ex>[rr]^{r_{11}} \ar@<1ex>[rr]^{r_{12}}\ar@<-3.5ex>[rr]^{\vdots}_{r_{1n}}&& \circ^2\ar@<4ex>[rr]^{r_{21}} \ar@<1ex>[rr]^{r_{22}}\ar@<-3.5ex>[rr]^{\vdots}_{r_{2n}}&&\circ^3 &\cdots&\circ^{n-1} \ar@<4ex>[rr]^{r_{(n-1)1}} \ar@<1ex>[rr]^{r_{(n-1)2}}\ar@<-3.5ex>[rr]^{\vdots}_{r_{(n-1)n}}&& \circ^n}.\]

For any path $\alpha=r_{si_s}\circ r_{(s+1)i_{s+1}}\circ \cdots \circ r_{ti_{t}}$, where $1\leq s<t\leq n-1$ and $1\leq i_j\leq n, s\leq j\leq t$, we associate a monomial $X^\alpha=x_{i_s}x_{i_{s+1}}\cdots x_{i_{t}}$ to $\alpha$ in $\C[x_1,\dots,x_n]$. Let $A=k\vec{Q}/I$, where $I=\langle\alpha-\beta|X^\alpha=X^\beta \mbox{ in } \C[x_1,\dots,x_n]\rangle$.
A. A. Beilinson proves that $\cd^b(A)$ is equivalent to $\cd^b(\P^{n-1})$ in \cite{Beilinson79}. By Theorem \ref{main theorem 1}, we know that the Coxeter polynomial of $A$
is $(\lambda+(-1)^{n-1})^n$.
\end{example}

For a fixed variety, it is very difficult to find a finite dimensional algebra which is derived equivalent to it. L. Costa and R. M. Mir\'{o}-Roig proved that the following toric varieties have tilting bundles whose summands are line bundles, see \cite{CM04}.

(i) Any $d$-dimensional smooth complete toric variety with splitting fan.

(ii) Any $d$-dimensional smooth complete toric variety with Picard number $\leq2$.

(iii) del Pezzo toric surfaces.

L. Hille and M. Perling prove that any smooth complete rational surface which can be obtained by blowing up a Hirzebruch surface $\F_a$ two-times in several points, also has a tilting bundle whose summands are line bundles. See \cite{ HP08, Per03, Per09}.

\begin{example}\label{example rational surface}
Let $X$ be a variety which can be obtained from $\P^2$ by blowing up one-time in $t$ points.
Let $\vec{Q}$ be the following quiver: 
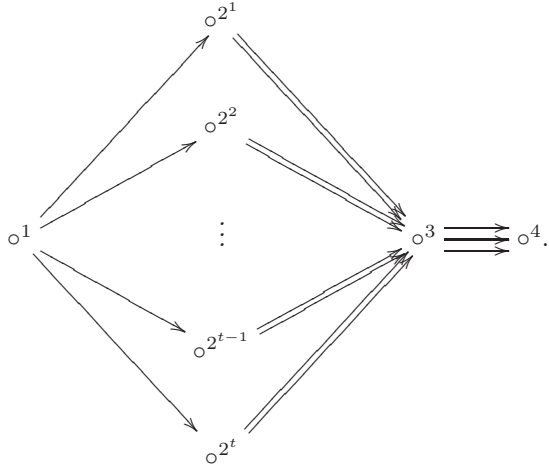
\[\xymatrix{&&\circ^{2^1} \ar@<0.25ex>[ddrr]\ar@<-0.25ex>[ddrr]&& &\\
&&\circ^{2^2} \ar@<0.25ex>[drr]\ar@<-0.25ex>[drr]&& &\\
\circ^1 \ar[uurr]\ar[urr]\ar[drr]\ar[ddrr]&&\vdots &&\circ^{3}\ar@<-1ex>[r]\ar@<1ex>[r]\ar[r] &\circ^4. \\
&&\circ^{2^{t-1}} \ar@<0.25ex>[urr]\ar@<-0.25ex>[urr]&& & \\
&&\circ^{2^t}  \ar@<0.25ex>[uurr]\ar@<-0.25ex>[uurr]&& &\\  }\]
When $t=0$, there are just three arrows from vertex $1$ to vertex $3$.
When $t=1$, there is additionally one arrow $b$ from vertex $1$ to vertex $3$, at this time, the quiver is as the following diagram shows.
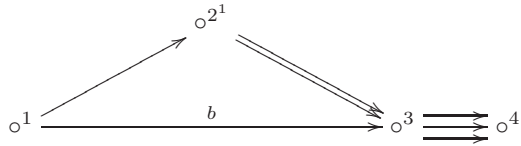
\[\xymatrix{&&\circ^{2^1}\ar@<0.25ex>[drr]\ar@<-0.25ex>[drr]&& &\\
\circ^1\ar[urr]\ar[rrrr]^b &&&& \circ^{3}\ar@<-1ex>[r]\ar@<1ex>[r]\ar[r] &\circ^4}\]

From \cite{Per09}, we know that $\cd^b(X)$ is equivalent to $\cd^b(A)$, where $A=k\vec{Q}/I$, $I$ is described in Section $7$ in \cite{Per09}.
By Theorem \ref{main theorem 1}, we know that the Coxeter polynomial of $A$ is $(\lambda+1)^{t+3}$ since the rank of the Grothendieck group of $A$ is $t+3$.
\end{example}

\section{The Jordan canonical form of the Coxeter transformation}
We prove in the previous section that for
a smooth complete variety $X$, the Coxeter transformation of
$X$ is a linear transformation with eigenvalues $1$ or $-1$. In this section, we compute its
Jordan canonical form. It is also an invariant under derived
equivalences. \\

For an integer $s\geq1$ and an element $\alpha\in \C$ or $k$, let
\[J(\alpha,s)=\left(\begin{array}{ccccccc}\alpha &1&&\\
&\ddots&\ddots&\\
&&\alpha&1\\
&&&\alpha\end{array}\right).\] which is the
\emph{Jordan block} of dimension $s$ with eigenvalues $\alpha$.

\subsection{Further Preliminaries}
First, we recall some definitions and facts about Fano varieties and toric varieties.
\begin{definition} (see, e.g., \cite{CLS})
A complete normal variety $X$ is said to be a \emph{Gorenstein Fano variety} if the anticanonical divisor $-K_X$ is Cartier and ample.
\end{definition}

Thus Gorenstein Fano varieties are projective. \\

In this subsection, we always assume that $N$ is a lattice of rank $n$.

\begin{definition} (\cite{CLS, Ful93})
A \emph{toric variety} is an irreducible normal variety $X$ containing a torus $T_N\simeq (\C^*)^n$ as a Zariski open subset such that the multiplication of $T_N$ on itself extends to an algebraic action of $T_N$ on $X$.
\end{definition}

By the theory of toric varieties, such a toric variety $X$ is characterized by a fan $\Sigma:=\Sigma(X)$ of strongly convex polyhedral cones in $N_\R$ if $X$ is normally separated.
\begin{definition} (\cite{CLS, Ful93})
A \emph{Convex polyhedral cone} in $N_\R$ is a set of the form
$$\sigma=\mathrm{Cone}(S)=\{\sum_{u\in S}\lambda_uu\,|\,\lambda_u\geq0\}\subseteq N_\R,$$
where $S\subseteq N_\R$ is finite. A convex polyhedral cone $\sigma$ is called \emph{strongly convex} if $\sigma$ contains no linear subspace. It is called rational if $S\subseteq N$.
\end{definition}
In this paper, we always consider strongly convex rational polyhedral cones and call it cone for simplicity.
\begin{definition}
A \emph{fan} in $N_\R$ is a finite collection of cones $\sigma\subseteq N_\R$ such that:

$(a)$ Every $\sigma\in\Sigma$ is a strongly convex rational polyhedral cone.

$(b)$ For all $\sigma\in\Sigma$, each face of $\sigma$ is also in $\Sigma$.

$(c)$ For all $\sigma_1,\sigma_2\in\Sigma$, the intersection $\sigma_1\cap \sigma_2$ is a face of each.
\end{definition}
Let $M:=\Hom_\Z(N,\Z)$ be the dual lattice. For any cone $\sigma$, denote $\sigma^\vee$ the dual cone in $M_\R$ by $$\sigma^\vee=\{m| \langle m,u \rangle\geq0,\mbox{ for any }u\in\sigma\}.$$
$\sigma^\vee$ is also strongly convex if $\sigma$ is strongly convex.

$m=(m_1,\dots,m_n)\in M$ gives a character $\chi^m: T_N\rightarrow \C^*$ defined by
$$\chi^m(t_1,\dots,t_n)=t_1^{m_1}\cdots t_n^{m_n}.$$ In this way, $M$ is the character lattice of $T_N$. So for any semigroup $S\subseteq M$, its \emph{semigroup algebra} is
$$\C[S]=\{\sum_{m\in S}c_m\chi^m\,|\,c_m\in \C\mbox{ and }c_m=0\mbox{ for all but finitely many }m\},$$
with multiplication induced by
$$\chi^m\cdot\chi^{m'}=\chi^{m+m'}.$$

A face $\tau$ of $\sigma$ is given by $\tau=\sigma\cap H_m$, where $m\in\sigma^\vee$ and $H_m=\{u\in N_\R|\langle m,u\rangle=0\}$ is the hyperplane defined by $m$. Let $S_\sigma=\sigma^\vee\cap M$ be the semigroup. Then $S_\tau=S_\sigma+\Z(-m)$ and $\C[S_\tau]$ is the localization $\C[S_\sigma]_{\chi^m}$.

To every cone $\sigma\in\Sigma$, there is an affine variety $U_\sigma:=\Spec(\C[\sigma^\vee\cap M])$ . If $\tau$ is a face of $\sigma$, then we have a natural embedding $U_\tau\hookrightarrow U_\sigma$. Now consider the collection of affine toric varieties $U_\sigma=\Spec(\C[S_\sigma])$, where $\sigma$ runs over all cones in a fan $\Sigma$. Let $\sigma_1$ and $\sigma_2$ be any two of these cones and let $\tau=\sigma_1\cap \sigma_2$. We have an isomorphism
$$g_{\sigma_2,\sigma_1}:(U_{\sigma_1})_{\chi^m}\simeq (U_{\sigma_2})_{\chi^{-m}}$$
which is the identity on $U_\tau$.
Define $X_\Sigma$ to be the variety glued by the affine varieties $U_\sigma$ along the subvarieties $(U_\sigma)_{\chi^m}$, see  (\cite{BG08, CLS, Ful93}).\\

The variety $X_\Sigma$ is a normally separated toric variety for a fan $\Sigma$ in $N_\R$. If $\Sigma$ is complete, then $X_\Sigma$ is compact in the classical topology.

Let $$\sigma^\perp=\{m| \langle m,u \rangle=0,\mbox{ for any }u\in\sigma\}.$$
For any cone $\sigma$ in $N_\R$, points of $U_\sigma$ are in bijective correspondence with semigroup homomorphisms $\gamma: S_\sigma\rightarrow \C$. For each cone $\sigma$, we have a point of $U_\sigma$ defined by
$$m\in S_\sigma\mapsto \left\{ \begin{array}{ccc}1 &m\in S_\sigma\cap \sigma^\perp=\sigma^\perp\cap M\\
0&\mbox{otherwise.}  \end{array} \right.$$
We denote this point by $\gamma_\sigma$ and call it the \emph{distinguished point} corresponding to $\sigma$. Now, by the $T_N$-action on $X_\Sigma$, we see that each cone $\sigma\in\Sigma$ corresponds to a torus orbit
$$O(\sigma)=T_N\cdot \gamma_\sigma\subseteq X_\Sigma.$$
It is known that $O(\sigma)\cong \Hom_\Z(\sigma^\perp\cap M,\C^*)\cong T_{N(\sigma)}$, where $N(\sigma)=N/N_\sigma$ and $N_\sigma$ is the sublattice of $N$ spanned by the points in $\sigma\cap N$, (\cite{CLS, Ful93}).

For a fan $\Sigma$ in $N_\R$, we denote by $\Sigma(i)$ the set of the cones of dimension $i$ in $\Sigma$.
So for any ray $\rho\in\Sigma(1)$, we get a prime $T_N$-invariant divisor $D_\rho=\overline{O(\rho)}$ on $X_\Sigma$.
\begin{example}
$(1)$ Consider the fan $\Sigma$ in $N_\R = \R^2$ in \emph{Figure}-$1$, where $N = \Z^2$ has
standard basis $e_1,e_2$. Then $X_\Sigma\cong\P^2$ and the divisor $D_{u_1}=D_{u_2}=D_{u_0}$ which is also the hyperplane divisor $H$.

$(2)$ Consider the fan $\Sigma$ in $N_\R = \R^2$ in \emph{Figure}-$2$. Then $X_\Sigma\cong \F_a$ which is the Hirzbruch surface if $a\geq0$; in particular, $X_\Sigma\cong \P_1\times\P_1$ if $a=0$. We also get that
the divisors $D_{u_1}=D_{u_3}$, $D_{u_2}=D_{u_4}-aD_{u_3}$, so $D_3,D_4$ is a basis of the Picard group. In the following, we denote $D_{u_3}$ and $D_{u_4}$ by $P$ and $Q$ respectively.
\begin{figure}[h]
\begin{center}
\includegraphics{rationalsurface.1}
\end{center}
\end{figure}
\end{example}

If $X_\Sigma$ is complete and simplicial, then $\CH^\ast(X_\Sigma)_\Q\cong  H^\ast(X_\Sigma,\Q)$ as rings, \cite{BG08, Ful93}.
Let $b_i$ be the rank of $\CH^i(X_\Sigma)$. Then $b_i$ is the $2i$-th Betti number of $X_\Sigma$.

Denote the set of the maximal cones in a fan $\Sigma$ by $\Sigma_{max}$.
\begin{lemma} (\cite{BG08, Ful93})\label{lemma toric variety}
Let $X_\Sigma$ be a smooth complete toric variety. Then the groups $\CH^\ast(X_\Sigma)$ and $K_0(X_\Sigma)\cong G_0(X_\Sigma)$ are free abelian groups of rank
$|\Sigma_{max}|$.
\end{lemma}
\begin{theorem}[hard Lefchetz theorem](\cite{BBD82, GM83}) Let $X$ be a normal projective variety of dimension $n$ and $\omega\in H^2(X,\Q)$ be the class of an ample divisor on $X$. Then we have the following:

$(1)$ For all $0\leq i\leq n$, multiplication by $\omega^{n-i}$ defines an isomorphism $$\omega^{n-i}: H^i(X)_\Q\rightarrow  H^{2n-i}(X)_\Q.$$

$(2)$ For $0\leq i< n$, multiplication by $\omega$
$$\omega: H^i(X)_\Q\rightarrow  H^{i+2}(X)_\Q$$ is injective.
\end{theorem}

It is easy to see that (2) follows from (1). Generally, for any variety $X$, an element in $ H^2(X,\Q)$ which satisfies condition (1) of the hard Lefchetz theorem is called a \emph{Lefschetz element}.\\

\subsection{Toric varieties}

We have a simple observation.
\begin{lemma}\label{lemma nilpotent matrix}
Let $A$ be an $m\times m$ nilpotent matrix over $\C$. Then the number and the dimensions of the Jordan blocks in the Jordan canonical form of $\mathrm{exp}(A)$ are the same as $A$, where $\mathrm{exp}(A)=\sum_{i=0}^\infty \frac{A^i}{i!}$.
\end{lemma}
\begin{proof}
$A$ is a nilpotent matrix, then there exists a $m\times m$ invertible matrix $P$ such that $PAP^{-1}=B=\mathrm{diag}(J_1,\dots,J_s)$, where $J_i$'s are the Jordan blocks of $A$, and the eigenvalues of $A$ are zeros. $P\mathrm{exp}(A)P^{-1}=\mathrm{exp}(PAP^{-1})=\mathrm{exp}(B)=\mathrm{diag}(B_1,\dots,B_s)$, where $B_i=\mathrm{exp}(J_i)$ for $1\leq i\leq s$. So we only need prove that $B_i$ has only one Jordan block.

It is easy to see that
$$B_i=\left(\begin{array}{ccccccc}1 &1&\frac{1}{2}&\cdots&\frac{1}{(j_i-1)!}\\
                                  &  1& 1&\cdots& \frac{1}{(j_i-2)!}\\
                                  &&1&\cdots&\frac{1}{(j_i-3)!}\\
                                  &&&\ddots&\vdots\\
                                  &&&&1
                                    \end{array}\right),$$
where $j_i$ is the dimension of $B_i$. All the eigenvalues of $B_i$ are $1$, and it is easy to see that $B_i-I_{j_i}$ has rank $j_i-1$, where $I_{j_i}$ is the identity matrix of dimension $j_i$. So $B_i$ has only one eigenvector $(1,0,\dots,0)$, the proof of this lemma is complete.
\end{proof}

Now we can prove our main result.

\begin{theorem}\label{main theorem 3}
Let $X_\Sigma$ be an n-dimensional smooth complete toric variety with ample canonical divisor or anticanonical divisor. Let $b_i$ be the rank of $\CH^{i}(X_\Sigma)$ for $0\leq i\leq n$ and $b_{-1}=0$. Then we have the following:

$(a)$ The number of Jordan blocks of the Jordan canonical form of the Coxeter transformation is $b_{[\frac{n}{2}]}$.

$(b)$ The maximal dimension of the Jordan blocks in the Jordan canonical form of the Coxeter transformation is $n+1$.

$(c)$ The Jordan canonical form of the Coxeter transformation has $b_i-b_{i-1}$ Jordan blocks of dimension $n-2i+1$ for any $0\leq i\leq [\frac{n}{2}]$, and these are all its Jordan blocks.

$(d)$ If the Jordan canonical form of the Coxeter transformation has $t_i$ $(t_i>0)$ Jordan blocks of dimension $d_i$ for $1\leq i\leq s$ and $d_1> d_2>\cdots >d_s$, then

\quad $(1)$ $t_1=1$, $n=d_1-1$.

\quad $(2)$
\begin{equation*}\left\{
\begin{split}
b_{\frac{d_1-d_i}{2}} &=b_{\frac{d_1-d_i+2}{2}} =\cdots =b_{\frac{d_1-d_{i+1}-2}{2}}=t_1+\cdots+t_i\quad \mbox{if }1\leq i\leq s-1, \\
b_{\frac{d_1-d_s}{2}} &=b_{\frac{d_1-d_s+2}{2}} =\cdots =b_{[\frac{n}{2}]}=t_1+\cdots+t_s\quad \mbox{if }i=s.
\end{split}\right.
\end{equation*}
\end{theorem}

\begin{proof}
For simplicity, we denote $K_{X_\Sigma}$ by $K_X$, which is the canonical divisor of $X_\Sigma$.

We only need consider the case that the canonical divisor $K_X$ is ample, since in the case that the anticanonical divisor is ample, we can consider the inverse transformation of the Coxeter transformation and the Jordan normal form of a transformation with eigenvalues $1$ is the same as its inverse.

Since $X_\Sigma$ is complete and simplicial, we get that $ H^{2i}(X_\Sigma,\Q)\cong \CH^i(X_\Sigma,\Q)$ and $ H^{2i+1}(X_\Sigma,\Q)=0$ for $0\leq i\leq n$. So by hard Lefschetz theorem, we know that $K_{X}^{n-2i}:\CH^i(X_\Sigma)_\Q\rightarrow \CH^{n-i}(X_\Sigma)_\Q$ is isomorphic for $0\leq i<[\frac{n}{2}]$.\\

Denote by $\theta: \CH^\ast(X_\Sigma)_\Q\rightarrow \CH^\ast(X_\Sigma)_\Q$ the linear transformation induced by the intersection product with $K_X$. Then $\theta$ is a nilpotent transformation and the Coxeter transformation is $\mathrm{exp}(\theta)$. By Lemma \ref{lemma nilpotent matrix}, we only need prove the results for $\theta$.\\

We know that $\CH^0(X_\Sigma)_\Q$ is one-dimensional.
Let $x^0_1$ be a generator of $\CH^0(X_\Sigma)_\Q$. Then the subspace spanned by $\{\theta^ix^0_1|0\leq i\leq n\}$ is an invariant subspace of $\theta$, its dimension is $n+1$. Since $\theta^ix^0_1\in\CH^i(X_\Sigma)$ and $\CH^i(X_\Sigma)\cap \CH^j(X_\Sigma)=0$ for any $i\neq j$, it is easy to see that $\theta^nx^0_1$ is an eigenvector of $\theta$ and $\{\theta^ix^0_1|0\leq i\leq n\}$ is the longest chain which contains $\theta^nx^0_1$, so the Jordan block corresponding to $\theta^nx^0_1$ is of dimension $n+1$.\\

From hard Lefschetz theorem, we know that $\theta:\CH^{n-1}(X_\Sigma)_\Q\rightarrow \CH^n(X_\Sigma)_\Q$ is surjective, so $\Im(\theta|_{\CH^{n-1}(X_\Sigma)_\Q})=\mathrm{Span}(\theta^n x^0_1):=\Q(\theta^n x^0_1)$. Since $\dim \CH^{n}(X_\Sigma)_\Q=1$, we can choose a basis $x^{n-1}_1,\dots,x^{n-1}_{b_{n-1}-1}$ of $\ker(\theta|_{\CH^{n-1}(X_\Sigma)_\Q})$, they are also eigenvectors of $\theta$. Since $\theta^{n-2}:\CH^1(X_\Sigma)_\Q\rightarrow \CH^{n-1}(X_\Sigma)_\Q$ is an isomorphism, there exist unique $x^1_1,\dots,x^1_{b_1-1}$ in $\CH^1(X_\Sigma)_\Q$ such that $\theta^{n-2}(x^1_{i})=x^{n-1}_i$ for $1\leq i\leq (b_1-1)$.
Thus for any $1\leq i\leq (b_1-1)$, $\{\theta^j x^1_i|0\leq j\leq (n-2) \}$ spans an invariant subspace of $\theta$, its dimension is $n-1$ since $\theta^j x^1_i$ are linear independent for fixed $i$. It is easy to see that $\{\theta^j x^1_i|0\leq j\leq (n-2) \}$ is also the longest chain which contains $x^{n-1}_i$. So we get that the canonical form of $\theta$ has $b_1-b_0$ Jordan blocks of dimension $n-1$. Do it inductively, one get that the canonical form of $\theta$ has $b_i-b_{i-1}$ Jordan blocks of dimension $n-2i+1$ for any $0\leq i\leq [\frac{n}{2}]$. To prove that these are all its Jordan blocks, we only need check the following equality:
$$\sum_{i=0}^{[\frac{n}{2}]}(b_i-b_{i-1})(n-2i+1)=\sum_{i=0}^n b_i.$$\\
Notice that $b_i=b_{n-i}$ from the Poincare duality or hard Lefschetz theorem.

Case 1: $n$ is odd. The left hand side is equal to $\sum_{i=0}^{[\frac{n}{2}]}2b_i$. The right hand side is equal to $$(n-2[\frac{n}{2}]+1)b_{[\frac{n}{2}]}+\sum_{i=0}^{[\frac{n}{2}]-1}((n-2i+1)-(n-2(i-1)+1))b_i=
2b_{[\frac{n}{2}]}+\sum_{i=0}^{[\frac{n}{2}]-1}2b_i$$ since $b_{-1}=0$. The equality is valid.

Case 2: $n$ is even. This case is similar to Case 1. The left hand side is equal to $\sum_{i=0}^{[\frac{n}{2}]-1}2b_i+b_{[\frac{n}{2}]}$. The right hand side is equal to $$(n-2[\frac{n}{2}]+1)b_{[\frac{n}{2}]}+\sum_{i=0}^{[\frac{n}{2}]-1}((n-2i+1)-(n-2(i-1)+1))b_i=
b_{[\frac{n}{2}]}+\sum_{i=0}^{[\frac{n}{2}]-1}2b_i$$ since $b_{-1}=0$. The equality is valid. So we get (c). Now, (a) and (b) are trivial.\\

For (d), if the Jordan canonical form of the Coxeter transformation has $t_i$ Jordan blocks of dimension $d_i$ for $1\leq i\leq s$ and $d_1> d_2>\cdots> d_s$, then from (a), (b) and (c), it is easy to see that $t_1=1$, $n=d_1-1$. For the last statement, by (c), we know that $d_i=n-2j_i+1$ ($1\leq i\leq s$) for some integer $j_i\geq0$. The formula in (2) make sense since $d_1-d_i$ is even for any $1\leq i\leq s$. If the Betti numbers $b_i$ are as (2) shows, then we can check directly that the Jordan canonical form of the Coxeter transformation has $t_i$ $(t_i>0)$ Jordan blocks of dimension $d_i$. So we only need prove that the Betti numbers are unique.
If there exist $b'_i$ such that they also satisfy the requirement, then $b_0=b'_0=1$. We denote
$$k=\max\{j|b_i=b'_i, \mbox{for any } 1\leq i\leq j\leq [\frac{n}{2}]\},$$ where $n=d_1-1$. So $k\geq 1$. If $k\neq [\frac{n}{2}]$, then either $b_{k+1}\neq b_k$ or $b'_{k+1}\neq b'_k$. For simplicity, we assume that $b_{k+1}\neq b_k$. Then the Jordan canonical form has $b_{k+1}-b_k$ Jordan blocks of dimension $n-2k-2$. By (c), we must also get that $b'_{k+1}-b'_k=b_{k+1}-b_k$. But $b_k=b'_k$, so $b'_{k+1}=b_{k+1}$. Contradict to our assumption. So the Betti numbers $b_i$ are unique.

The proof is complete.
\end{proof}
\begin{remark}\label{remark lefschetz elment}
From the proof, we know that the above theorem is also valid if the canonical bundle of a smooth complete toric variety is a Lefschetz element. So an interesting question is when this condition holds.
\end{remark}

The following example shows that there exist smooth complete varieties of dimension greater than $2$ with neither ample anonical nor ample anticanonical divisor also have the same properties as Theorem \ref{main theorem 2} says.
\begin{example}
Let $X$ be the blow-up of $\P^3$ at two fixed points under the action of the torus. It is easy to check that neither the canonical nor the anticanonical divisor of $X$ is ample. Notice that $X$ is also a toric variety with fan $\Sigma$ generated by the columns of the following matrix
\[\left(\begin{array}{cccccccc} 1&0&0&-1&-1&0\\
0&1&0&-1&0&-1\\
0&0&1&-1&0&0  \end{array}\right).\]
Denote the cone generated by the $i$-th column by $\tau_i$. Then the maximal cones in $\Sigma$ are $\mathrm{Cone}(\tau_1,\tau_2,\tau_3)$, $\mathrm{Cone}(\tau_1,\tau_2,\tau_4)$, $\mathrm{Cone}(\tau_2,\tau_4,\tau_5)$, $\mathrm{Cone}(\tau_3,\tau_4, \tau_5)$, $\mathrm{Cone}(\tau_2,\tau_3,\tau_5)$, $\mathrm{Cone}(\tau_1,\tau_3,\tau_6)$, $\mathrm{Cone}(\tau_3,\tau_4,\tau_6)$ and $\mathrm{Cone}(\tau_1,\tau_4,\tau_6)$.

Let $D_i$ be the divisor corresponding to the one-dimensional cone generated by the $i+1$ column for $0\leq i\leq 5$.
Then the Chow group $\CH^\ast(X)_\Q$ is generated by $D_3,D_4,D_5$  with the relations $D_3D_4+D_4^2=0$, $D_3D_5+D_5^2=0$, $D_4D_5=0$, $D_3^3+D_3^2D_5=0$, $D_3^3+D_3^2D_4=0$.
So $X, D_3,D_4,D_5, D_3^2,D_4^2,D_5^2, D_3^3$ is a basis of the Chow group.
It is easy to get that $-K_X=4D_3+2D_4+2D_5$. One can check that the Jordan canonical form of the Coxeter transformation is $J(1,4)\oplus J(1,2)\oplus J(1,2)$.
\end{example}

\begin{lemma} (\cite{CLS})\label{lemma Betti number}
Let $\Sigma$ be a complete simplicial fan in $N_\R\cong\R^n$. Then the Betti numbers of $X_\Sigma$ are given by
$$b_k(X_\Sigma)=\sum_{i=k}^n(-1)^{i-k}\left(\begin{array}{cc}i\\k \end{array} \right)|\Sigma(n-i)|,\quad(4\mbox{-}1)$$
where $b_k(X_\Sigma)$ is the $2k$-th Betti number.
\end{lemma}
Then we have the following corollary of Theorem \ref{main theorem 3}.
\begin{corollary}
Let $\Sigma$ be a complete smooth fan in $N_\R\simeq \R^n$. If $X_\Sigma$ is a variety with ample canonical divisor or anticanonical divisor, then
the Jordan canonical form of the Coxeter transformation has exactly

$$\sum_{i=k-1}^n(-1)^{i-k}\left(\begin{array}{cc}i+1\\k \end{array} \right)|\Sigma(n-i)| $$
many Jordan blocks of dimension $n-2k+1$ for any $0\leq k\leq [\frac{n}{2}]$.
\end{corollary}
\begin{proof}
By Theorem \ref{main theorem 3} and Lemma \ref{lemma Betti number}, we know that the Jordan canonical form of the Coxeter transformation has exactly
$$
\left[\sum_{i=k}^n(-1)^{i-k}\left(\begin{array}{cc}i\\k \end{array} \right)|\Sigma(n-i)|\right]-\left[\sum_{i=k-1}^n(-1)^{i-k+1}\left(\begin{array}{cc}i\\k-1 \end{array} \right)|\Sigma(n-i)|\right] \quad(4\mbox{-}2)
$$
Jordan blocks of dimension $n-2k+1$ for any $0\leq k\leq [\frac{n}{2}]$.
Then (4-2) is equal to
\begin{eqnarray*}
&\left[\sum_{i=k}^n(-1)^{i-k}\left(\begin{array}{cc}i\\k \end{array} \right)|\Sigma(n-i)|\right]-\left[\sum_{i=k}^n(-1)^{i-k+1}\left(\begin{array}{cc}i\\k-1 \end{array} \right)|\Sigma(n-i)|\right]\hspace{2.5cm}\\
&-|\Sigma(n-k+1)| \hspace{11.6cm}\\
=&\left\{\sum_{i=k}^n(-1)^{i-k}\left[\left(\begin{array}{cc}i\\k \end{array} \right)+ \left(\begin{array}{cc}i\\k-1 \end{array} \right) \right]|\Sigma(n-i)|\right\}-|\Sigma(n-k+1)|\hspace{3.9cm}\\
=&\left[  \sum_{i=k}^n(-1)^{i-k}\left(\begin{array}{cc}i+1\\k \end{array} \right)|\Sigma(n-i)|  \right] -|\Sigma(n-k+1)|\hspace{5.9cm}\\
=&\sum_{i=k-1}^n(-1)^{i-k}\left(\begin{array}{cc}i+1\\k \end{array} \right)|\Sigma(n-i)|. \hspace{8.3cm}
\end{eqnarray*}
\end{proof}

In the following, we compute the Jordan canonical form of the Coxeter transformation for lower dimensional toric Fano varieties.

From \cite{Oda88}, we know that there are five toric Fano surface, also called toric Del Pezzo surfaces. We give a list of the Jordan blocks of the Coxeter transformations in the following.
\[\begin{tabular}{|l|l|}
\hline
\multicolumn{1}{|c|}{Varieties}
&\multicolumn{1}{|c|}{Jordan blocks}\\
\hline
\multicolumn{1}{|c|}{$\P^2$}
&\multicolumn{1}{|c|}{3}\\
\hline
\multicolumn{1}{|c|}{$\P^1\times\P^1$}
&\multicolumn{1}{|c|}{3,1}\\
\hline
\multicolumn{1}{|c|}{$\P^2$ blown up in one point (the Hirzebruch surface $\F_1$)}
&\multicolumn{1}{|c|}{3,1}\\
\hline
\multicolumn{1}{|c|}{$\P^2$ blown up in two points}
&\multicolumn{1}{|c|}{3,1,1}\\
\hline
\multicolumn{1}{|c|}{$\P^2$ blown up in three points}
&\multicolumn{1}{|c|}{3,1,1,1}\\
\hline
\end{tabular}
\]

From \cite{Batyrev82} and \cite{Wat82}, we know that there are 18 toric fano three-fold. We give a list of the Jordan blocks of the Coxeter transformations in the following.
\[\begin{tabular}{|l|l|l|}
\hline
\multicolumn{1}{|c|}{Type}
&\multicolumn{1}{|c|}{Varieties}
&\multicolumn{1}{|c|}{Jordan blocks}\\
\hline
\multicolumn{1}{|c|}{Type I}
&\multicolumn{1}{|c|}{$\P^3$}
&\multicolumn{1}{|c|}{4}\\ \hline
&\multicolumn{1}{|c|}{$\P(\co_{\P^2}\oplus\co_{\P^2}(2))$}&
\multicolumn{1}{|c|}{4,2}\\
&\multicolumn{1}{|c|}{$\P(\co_{\P^2}\oplus\co_{\P^2}(1)$}
&\multicolumn{1}{|c|}{4,2}\\
\multicolumn{1}{|c|}{Type II}
&\multicolumn{1}{|c|}{$\P(\co_{\P^1}\oplus \co_{\P^1}\oplus\co_{\P^1}(1) $}
&\multicolumn{1}{|c|}{4,2}\\
&\multicolumn{1}{|c|}{$\P(\co_{\P^1\times\P^1}\oplus \co_{\P^1\times\P^1}(1,1)$}
&\multicolumn{1}{|c|}{4,2,2}\\
&\multicolumn{1}{|c|}{$\P(\co_{\P^1\times\P^1}\oplus \co_{\P^1\times\P^1}(1,-1)$}
&\multicolumn{1}{|c|}{4,2,2}\\
&\multicolumn{1}{|c|}{$\P(\co_{\widetilde{\P^2}(1)}\oplus \co_{\widetilde{\P^2}(1)}(l))$}
&\multicolumn{1}{|c|}{4,2,2}\\
\hline
&\multicolumn{1}{|c|}{$\P^2\times\P^1$}&
\multicolumn{1}{|c|}{4,2}\\
&\multicolumn{1}{|c|}{$\P^1\times\P^1\times\P^1$}
&\multicolumn{1}{|c|}{4,2,2}\\
\multicolumn{1}{|c|}{Type III}
&\multicolumn{1}{|c|}{$\widetilde{\P^2}(1)\times\P^1$}
&\multicolumn{1}{|c|}{4,2,2}\\
&\multicolumn{1}{|c|}{$\widetilde{\P^2}(2)\times\P^1$}
&\multicolumn{1}{|c|}{4,2,2,2}\\
&\multicolumn{1}{|c|}{$\widetilde{\P^2}(3)\times\P^1$}
&\multicolumn{1}{|c|}{4,2,2,2,2}\\
\hline
&\multicolumn{1}{|c|}{Blow up of $\P^1$ on $\P(\co_{\P^2}\oplus\co_{\P^2}(2))$}&
\multicolumn{1}{|c|}{4,2,2}\\
&\multicolumn{1}{|c|}{Blow up of $\P^1$ on $\P^2\times\P^1$}
&\multicolumn{1}{|c|}{4,2,2}\\
\multicolumn{1}{|c|}{Type IV}
&\multicolumn{1}{|c|}{$\widetilde{\P^2}(2)$-bundle over $\P^1$}
&\multicolumn{1}{|c|}{4,2,2,2}\\
&\multicolumn{1}{|c|}{$\widetilde{\P^2}(2)$-bundle over $\P^1$}
&\multicolumn{1}{|c|}{4,2,2,2}\\
&\multicolumn{1}{|c|}{$\widetilde{\P^2}(2)$-bundle over $\P^1$}
&\multicolumn{1}{|c|}{4,2,2,2}\\
&\multicolumn{1}{|c|}{$\widetilde{\P^2}(3)$-bundle over $\P^1$}
&\multicolumn{1}{|c|}{4,2,2,2,2}\\
\hline

\end{tabular}
\]
From \cite{Batyrev94, Sato00}, we know that there are 124 toric Fano 4-folds.\\

\subsection{Rational surfaces}

Now, we consider rational surfaces.

By the classification of rational
surfaces, for any given rational surface $X$, there exists a
sequence of blow-downs
$$X=X_t\xrightarrow{b_t}X_{t-1}\xrightarrow{b_{t-1}}\cdots\xrightarrow{b_2}X_1\xrightarrow{b_1}X_0,$$
where $X_0$ is either $\P^2$ or some Hirzebruch surface $\F_a$.

We fix such a sequence. Denote $R_1,\dots,R_t$ the classes of the
total transformations on $X$ of the exceptional divisors of the
blow-ups $b_i$. If $X_0\cong \P^2$, we get $H,R_1,\dots,R_t$ as a
basis of $\mathrm{Pic}(X)$ with the following intersection
relations:
$$H^2=1,H\cdot R_i=0, R_i^2=-1, R_i\cdot R_j=0\mbox{ for any }i, j\mbox{ and } i\neq j.\quad(4\mbox{-}3)$$
If $X_0\cong \F_a$, we get $P,Q,R_1,\dots, R_t$ as a basis of
$\mathrm{Pic}(X)$ with $$P^2=0,Q^2=a,P\cdot Q=1,R_i^2=-1,$$
$$R_i\cdot R_j=0, P\cdot R_i=Q\cdot R_i=0\mbox{ for any }i,j \mbox{ and }i\neq j.\quad(4\mbox{-}4)$$
See, e.g., \cite{Per09}.

The following lemma characterizes the canonical divisor and the Picard group of the blow-up.
\begin{lemma} (\cite{Hart77}) \label{lemma blow up}
$(1)$ Let $f:X\rightarrow Y$ be a dominating morphism between smooth
varieties such that $\dim X=\dim Y=n$. Then
$$K_X=f^*K_Y+R_f,$$ where $R_f\geq0$ is the ramification divisor of
$f$.

$(2)$ Let $f:\tilde{X}\rightarrow X$ be a blow-up of $X$ at a smooth
subvariety $Y$ of codimension $r\geq2$, and let $E$ be the
exceptional divisor of the blow up. Then
$$\mathrm{Pic}(\tilde{X})=f^*\mathrm{Pic}(X)\oplus\Z E$$ and
$$K_{\tilde{X}}=f^*K_X+(r-1)E.$$
\end{lemma}

\begin{lemma}\label{lemma rational surface}
Let $X$ be a rational surface. Then the canonical divisor is a Lefschetz element if and only if the rank of the Picard group is not $10$.
\end{lemma}
\begin{proof}
There exists a
sequence of blow-downs
$$X=X_t\xrightarrow{b_t}X_{t-1}\xrightarrow{b_{t-1}}\cdots\xrightarrow{b_2}X_1\xrightarrow{b_1}X_0.$$

Case 1 $X_0=\P^2$. By Lemma \ref{lemma blow up}, we
know $K_X=-3H+\sum_{i=1}^tR_i$ since the canonical divisor of $\P^2$
is $-3H$. From the intersection relations (4-3), we get that
$$K_X^2=(9-t)[pt].$$
But $\CH^i(X)_\Q=0$ for $i\neq 0,1,2$, it follows that $K_X$ is a Lefschetz element if and only if $K_X^2\neq 0$, which is equivalent to $t\neq 9$.
Furthermore, $H,R_1,\dots,R_t$ is a
basis of $\mathrm{Pic}(X)$, so the rank of the Picard group is $t+1$ and then the result follows.

Case 2 $X_0=\F_a$. Similarly, since the canonical divisor of $\F_a$
is $(a-2)P-2Q$, we
know $K_X=(a-2)P-2Q+\sum_{i=1}^tR_i$. From the intersection relations (4-4), we get that
$$K_X^2=(8-t)[pt].$$
It follows that $K_X$ is a Lefschetz element if and only if $t\neq 8$.
Furthermore, $P,Q,R_1,\dots,R_t$ is a
basis of $\mathrm{Pic}(X)$, so the rank of the Picard group is $t+2$ and then the result follows.
\end{proof}

\begin{theorem}\label{main theorem 2}
Let $X$ be a rational surface.

$(1)$ The number of the Jordan blocks in Jordan canonical form of the Coxeter transformation of $X$ is equal to $\mathrm{rank}(\mathrm{Pic}(X))$.

$(2)$ If $\mathrm{rank}(\mathrm{Pic}(X))\neq10$, then the Jordan canonical form of the Coxeter transformation has a Jordan block of dimension $3$, the others are of dimension $1$; otherwise, the Jordan canonical form of the Coxeter transformation has two Jordan blocks of dimension $2$, and eight Jordan blocks of dimension $1$.
\end{theorem}
\begin{proof}
By Lemma \ref{lemma rational surface} and Remark \ref{remark lefschetz elment}, we get (1) and the first statement of (2).

For the second statement of (2), the matrix which represents the Coxeter transformation is $12\times 12$. If $X$ is a blow-up of $\P^2$, by Lemma \ref{lemma rational surface}, we know $t=9$ and $\CH^1(X)_\Q$
has a basis $H, R_1,\dots,R_9$ since $\CH^1(X)\cong
\mathrm{Pic}(X)$. It is easy to see that $\CH^0(X)$ and $\CH^2(X)$ are
one-dimensional and generated by $[X]$ and $[pt]$ respectively. So $[X],
H, R_1,\dots,R_9,[pt]$ is a basis of $\CH^\ast(X)_\Q$. We denote by $\Psi$ the linear
transformation induced by $\cdot \ch_X(\omega_X)$ on the
$\CH^\ast(X)_\Q$. It is easy to see the Jordan blocks of $\Psi$ has the same dimension as the Jordan blocks of $\Phi$. By Lemma \ref{lemma blow up}, we
know $K_X=-3H+\sum_{i=1}^9R_i$ since the canonical divisor of $\P^2$
is $-3H$.
From the intersection relations, we can get
$$\ch_X(\omega_X)=1-3H+\sum_{i=1}^9R_i.$$
The
matrix of $\Psi$ under the above basis is as follows
\[\left(\begin{array}{cccccccccccccc} 1&-3&1&1&1&1&1&1&1&1&1&0\\
                                    &1 &0&0&0&0&0&0&0&0&0&-3\\
                                    &  &1&0&0&0&0&0&0&0&0&-1\\
                                    &   & & 1&0&0&0&0&0&0&0&-1\\
                                    &   & & &1&0&0&0&0&0&0&-1\\
                                    &   & & &&1&0&0&0&0&0&-1\\
                                    &   & & &&&1&0&0&0&0&-1\\
                                    &   & & &&&&1&0&0&0&-1\\
                                    &   & & &&&&&1&0&0&-1\\
                                    &   & & &&&&&&1&0&-1\\
                                    &   & & &&&&&&&1&-1\\
                                    &&&&&&&&&&&1\end{array}\right).\]
It is easy to check that this matrix has two Jordan blocks of dimension $2$, and eight Jordan blocks of dimension $1$.

If $X$ is a blow-up of $\F_a$, similarly, we get that $t=8$ and $X, P, Q, R_1,\dots,R_8,[pt]$ is a basis of $\CH^\ast(X)_\Q$. Also by Lemma \ref{lemma blow up}, since the canonical divisor of $\F_a$
is $(a-2)P-2Q$, we
know $K_X=(a-2)P-2Q+\sum_{i=1}^8R_i$.
From the intersection relations, we get $$\ch_X(\omega_X)=1+(a-2)P-2Q+\sum_{i=1}^8R_i.$$
The
matrix of $\Psi$ under the above basis is as follows
\[\left(\begin{array}{cccccccccccccc} 1&a-2&-2&1&1&1&1&1&1&1&1&0\\
                                    &1 &0&0&0&0&0&0&0&0&0&-2\\
                                    &  &1&0&0&0&0&0&0&0&0&-(a+2)\\
                                    &   & & 1&0&0&0&0&0&0&0&-1\\
                                    &   & & &1&0&0&0&0&0&0&-1\\
                                    &   & & &&1&0&0&0&0&0&-1\\
                                    &   & & &&&1&0&0&0&0&-1\\
                                    &   & & &&&&1&0&0&0&-1\\
                                    &   & & &&&&&1&0&0&-1\\
                                    &   & & &&&&&&1&0&-1\\
                                    &   & & &&&&&&&1&-1\\
                                    &&&&&&&&&&&1\end{array}\right).\]
It is easy to check that this matrix also has two Jordan blocks of dimension $2$, and eight Jordan blocks of dimension $1$.
So (2) is valid.
\end{proof}

\begin{corollary}\label{corollary derived equivalent}
If a rational surface or a smooth complete toric variety $X$ with ample canonical or anticanonical divisor is derived equivalent to a finite dimensional algebra $A$, then the Betti numbers of $X$ are determined by the matrix $-C_A^tC_A^{-1}$.
\end{corollary}
\begin{proof}
Since $X$ is derived equivalent to $A$, we know that the Jordan canonical form of the Coxeter transformation $\Phi_X$ is the same as that of $-C_A^tC_A^{-1}$.
From Theorem \ref{main theorem 3} and Theorem \ref{main theorem 2}, we know that the rank of $\CH^i(X)$ is determined by the matrix $-C_A^tC_A^{-1}$.
For $X$ is a rational surface or a smooth complete toric variety, we know that $b_i=\mathrm{rank}(\CH^i(X))=\mathrm{rank}(H^{2i}(X))$ is the $2i$-th Betti number. So the Betti numbers of $X$ are determined by the matrix $-C_A^tC_A^{-1}$.
\end{proof}
At the end of this section, we continue to consider the examples in Section 3.
\begin{example}\label{example projective space2}
From Theorem \ref{main theorem 3}, the Jordan canonical form of the Coxeter transformation of $\P^n$ has only one Jordan block. Let $A$ be the algebra in Example \ref{example projective space}. Then the Jordan canonical form of the Coxeter transformation $\Phi_A$ has only one Jordan block.
\end{example}
\begin{example}
Let $A$ be the algebra in Example \ref{example rational surface}. Then the Jordan canonical form of the Coxeter transformation $\Phi_A$ has $t+1$ Jordan blocks with only one of dimension $3$, and the others are of dimension $1$ if $t\neq 9$. Otherwise, $\Phi_A$ has $t+1$ Jordan blocks with two of dimension $2$, and the others are of dimension $1$ if $t=9$.
\end{example}

\section{Jordan canonical forms of tensors of matrices}
To construct graded local Frobenius algebras over an algebraically closed
field $k$, it is important to find out Jordan canonical forms of
tensors product of square matrices, see \cite{Wa03}. In fact, it is known that any graded local
Frobenius algebra is of the form of $\Lambda(\phi,\gamma)=T(V)/R(\phi,\gamma)$, where $V$ is a finite dimensional $k$-vector space, $\gamma\in \mathrm{GL}(V)$, and $\phi:V^{\otimes n}\rightarrow k$ a $k$-linear map satisfying several conditions. Further, if we decompose $V$ as $(V,\gamma)=\oplus_i(V_i,\gamma_i)$, where $V_i$ are invariant subspaces of $\gamma$ and $\gamma|_{V_i}=\gamma_i$, then the conditions of $\phi$ can be described in terms of each $\phi_{i_1\dots i_r}:V_{i_1}\otimes\cdots \otimes V_{i_r}\rightarrow k$. Then we have to consider a Jordan canonical form of $\gamma_{i_1}\otimes\cdots \otimes \gamma_{i_r}$ as an element in $\mathrm{GL}(V_{i_1}\otimes \cdots\otimes V_{i_r})$.\\

From tensor product of matrices, we can define the operator $\boxtimes$ as follows
$$J(\alpha,s)\boxtimes J(\beta,t):=\mbox{Jordan canonical form of }J(\alpha,s)\otimes J(\beta,t).$$

The explicit formula of this operator is as the following lemma shows.

\begin{lemma} (\cite{MV})\label{lemma Jordan block}
The following holds for integers $s\leq t$ and $\alpha,\beta\in k$:
\[J(\alpha,s)\boxtimes J(\beta,t)=\left\{\begin{array}{cccc}J(0,s)^{\oplus t-s+1}\oplus\bigoplus ^{2s-2}_{i=1}J(0,s-[\frac{i}{2}])& \mbox{ if }\alpha=0=\beta\\
J(0,s)^{\oplus t}& \mbox{ if }\alpha=0\neq\beta\\
J(0,t)^{\oplus s}&\mbox{ if }\alpha\neq0=\beta\\
\bigoplus^s_{i=1}J(\alpha\beta,s+t+1-2i)&\mbox{ if
}\alpha\neq0\neq\beta
\end{array}\right. .\]
\end{lemma}

In the following, we will generalize the above formula in Lemma \ref{lemma Jordan block}.
\begin{lemma}\label{lemma variety product}
Let $X$ and $Y$ be two smooth complete toric varieties and $\pi_X:X\times
Y\rightarrow X$, $\pi_Y:X\times Y\rightarrow Y$ be the projections.
Then

(1) $K_0(X\times Y)\cong K_0(X)\otimes_\Z K_0(Y)$. In particular,
$\mathrm{rank}(K_0(X\times Y))= \mathrm{rank}(K_0(X))\times
\mathrm{rank}(K_0(Y))$.

(2) $\omega_{X\times Y}\cong\pi_X^*\omega_X\otimes
\pi_Y^*\omega_Y$, the right hand side is also denoted by
$\omega_X\boxtimes\omega_Y$.
\end{lemma}
\begin{proof}
Because $X$ and $Y$ are two smooth complete toric varieties, from Lemma \ref{lemma toric variety} and $\CH^*(X)\cong H^*(X)$, $\CH^*(Y)\cong H^*(Y)$, we know $K_0(X)\cong  H^*(X)$ and $K_0(Y)\cong H^*(Y)$. From the K\"{u}nneth formula, we have
$$
 H^*(X\times Y)\cong  H^*(X)\otimes_\Z  H^*(Y).
$$
But $X\times Y$ is a smooth complete toric variety, so $K_0(X\times Y)\cong H^*(X\times Y)$.
Thus (1) is valid.

(2) is direct from the definition of the canonical sheaf.
\end{proof}

So we get the following proposition.
\begin{proposition}\label{theorem Jordan block}
Let $X$ and $Y$ be two smooth complete toric varieties.

(a) Let $\{a_i|i\in I, \mbox{ for some finite set }I\}$ and $\{b_j|j\in J, \mbox{ for some finite set }J\}$ be the eigenvalues of the Coxeter transformation of $X$ and $Y$ respectively. Then $\{-a_ib_j|i\in I,j\in J\}$ are the eigenvalues of the Coxeter transformation of $X\times Y$.

(b)If the Jordan
canonical form of the Coxeter transformations of $X$ and $Y$ are
$\bigoplus_{i=1}^{t_1}J^X_i$ and $\bigoplus_{j=1}^{t_2}J^Y_j$ respectively, then the sizes and number of Jordan blocks in
the Jordan canonical form of the Coxeter transformation of $X\times Y$ are the same as that of
$\bigoplus_{i,j=1}^{t_1,t_2}J^X_i\boxtimes J^Y_j$.
\end{proposition}
\begin{proof}
Denote $n=\dim X$ and $m=\dim Y$.
Let $\omega_X$ and $\omega_Y$ be the canonical bundles of $X$ and $Y$ respectively. By Lemma \ref{lemma variety product}, we know $$K_0(X\times Y)_\Q\cong K_0(X)_\Q\otimes K_0(Y)_\Q, \mbox{ and } \omega_{X\times Y}\cong\omega_X\boxtimes \omega_Y.\quad(5\mbox{-}1)$$
Let $\Phi_X$ and $\Phi_Y$ be the Coxeter transformations of $X$ and $Y$ respectively.
Then $\Phi_X$ and $\Phi_Y$ are induced by the functors $-\otimes \omega_X[n-1]$ and $-\otimes \omega_X[m-1]$ respectively. Furthermore, $\Phi_{X\times Y}$ is the Coxeter transformation of $X\times Y$ induced by $-\otimes \omega_{X\times Y}[m+n-1]$.
Also let $\Psi_X$ be the transformation induced by $-\otimes \omega_X$, and the others are similar.

For any smooth complete toric variety $W$, we know $K_0(W)\cong H^*(W)$ as rings, so for any $\alpha\in K_0(W)$, we also use $\alpha$ to denote the corresponding element in $H^*(W)$. We also use $\Phi_W$ and $\Psi_W$ to denote the induced transformation on $H^*(W)$.

From the K\"{u}nneth formula, we have an isomorphism of rings
$$
H^*(X)\otimes_\Z  H^*(Y)\cong H^*(X\times Y).
$$
which maps $\alpha\otimes \beta$ to the cochain cross product $\alpha\times \beta$. The ring structure on $H^*(X)\otimes_\Z  H^*(Y)$ is defined to be
$$(\alpha\otimes \beta)\cup(\alpha'\otimes \beta')=(-1)^{pq}(\alpha\cup\alpha')\otimes (\beta\cup\beta'),$$ if $\beta\in H^p(Y)$ and $\alpha'\in H^q(X)$.

Notice that $\Psi_{X\times Y}$ acts on $H^*(X\times Y)$ as $\cup (\pi_1^*\omega_X\cup \pi_2^*\omega_Y)$ and $\omega_X\times \omega_Y=\pi_1^*\omega_X\cup \pi_2^*\omega_Y$. So the induced transformation $\Psi_{X\times Y}$ acts on $H^*(X)\otimes_\Z  H^*(Y)$ as $\cup(\omega_X\otimes \omega_Y)$.
For any $\alpha\otimes \beta\in H^p(X)\otimes_\Z  H^q(Y)$, we get the induced transformation $\Psi_{X\times Y}$ on $H^*(X)\otimes_\Z  H^*(Y)$ as follows:
\begin{equation*}
\begin{split}
\Psi_{X\times Y}(\alpha\otimes \beta)&=(\alpha\otimes \beta)\cup (\omega_X\otimes \omega_Y)\\
&=(-1)^{pq}(\alpha\cup \omega_X)\otimes(\beta\cup\omega_Y)\\
&=(\alpha\cup \omega_X)\otimes(\beta\cup\omega_Y)\\
&=\Psi_X(\alpha)\otimes\Psi_Y(\beta).
\end{split}
\end{equation*}
The third equality holds since for any smooth complete toric variety $W$, if $H^i(W)\neq0$, then $i$ is even.

Therefore $-\Phi_X\otimes \Phi_Y$ is the Coxeter transformation of $X\times Y$ under the isomorphisms $(5\mbox{-}1)$.  But $-\Phi_X\otimes\Phi_Y$ has the same Jordan canonical form as $\Phi_X\otimes \Phi_Y$ except the eigenvalues. The proof is complete.
\end{proof}

In the following, we use Theorem \ref{main theorem 3} to generalize Lemma \ref{lemma Jordan block}. Let
$$\prod_{j=1}^tv_{r_j}=\bigoplus_{k=0}^{n}c_k x^k, \mbox{ where }n=-t+\sum_{j=1}^{t}r_j,v_{n}=1+x+\cdots+x^{n-1}.\quad(5\mbox{-}2)$$
Then we have the following proposition.
\begin{proposition}\label{proposition jordan block}
Let $\alpha_1,\dots,\alpha_t\neq0$ and $0<r_1\leq r_2\leq\cdots\leq r_t$ be $t$ integers. Then $$J(\alpha_1,r_1)\boxtimes J(\alpha_2,r_2)\boxtimes\cdots\boxtimes J(\alpha_t,r_t)=\bigoplus_{k=0}^{[\frac{n}{2}]}( J(\alpha, n+1-2k)^{\oplus (c_k-c_{k-1})}),$$ where $\alpha=\prod_{i=1}^t\alpha_i$ and $c_k$'s are given in formula $(5\mbox{-}2)$ and $c_{-1}=0$.
\end{proposition}
\begin{proof}
By Lemma \ref{lemma Jordan block}, do it inductively, we can see that $$J(\alpha_1,r_1)\boxtimes J(\alpha_2,r_2)\boxtimes\cdots\boxtimes J(\alpha_t,r_t)$$ is a finite direct sum of some Jordan blocks and their eigenvalues are equal to $\alpha$. Since the Jordan canonical form of the tensor product of these Jordan blocks does not depend on $\alpha_1,\dots,\alpha_t$ if they are all non-zero, we only need do it for $\alpha_i=1$ or $-1$ for $1\leq i\leq t$.

By Example \ref{example projective space2}, we know that the Jordan canonical form of the Coxeter transformation of $\P^{r_i-1}$ has only one Jordan block of dimension $r_i$ with eigenvalues $(-1)^{r_i}$.
From Proposition \ref{theorem Jordan block}, we also know that the Jordan canonical form of the Coxeter transformation of $X=\P^{r_1-1}\times\P^{r_2-1}\times\cdots \times\P^{r_t-1}$ is the same as $$J((-1)^{r_1},r_1)\boxtimes J((-1)^{r_2},r_2)\boxtimes\cdots\boxtimes J((-1)^{r_t},r_t),$$ except the eigenvalues.

From the K\"{u}nneth formula, we know that the $2k$-th Betti number of $X$ is $c_k$ where $c_k$ is as the formula (5-2) shows.
By Theorem \ref{main theorem 3}, we get that $$J(\alpha_1,r_1)\boxtimes J(\alpha_2,r_2)\boxtimes\cdots\boxtimes J(\alpha_t,r_t)$$ has $(c_k-c_{k-1})$ Jordan blocks of dimension $(n+1-2k)$ for $0\leq k\leq [\frac{n}{2}]$ since $n=-t+\sum_{j=1}^{t}r_j$ is the dimension of $X$.
\end{proof}

\begin{remark}
Let $A$ and $B$ be two finite dimensional algebras with finite global diemension over
algebraically closed field $k$. Then the relation between the Jordan
canonical form of the Coxeter transformations of $A$, $B$ and $A\otimes_k B$
is the same as Proposition \ref{theorem Jordan block} describes. In particular,
$$K_0(A\otimes_k B)\cong K_0(A)\otimes K_0(B).$$
\end{remark}

In fact, just as the proof of Proposition 4.16 in \cite{CM04}, we can prove the following.
\begin{proposition}\label{proposition deireved equivalent}
Let $X$, $Y$ be two smooth projective varieties over $k$ and $A$, $B$ be two finite
dimensional algebras with finite global dimensions also over $k$. If
$\cd^b(X)\simeq \cd^b(A)$ and $\cd^b(Y)\simeq \cd^b(B)$ induced by tilting bundles, then
$\cd^b(X\times Y)\simeq \cd^b(A\otimes B)$.
\end{proposition}
\begin{proof}
Let $E$ and $F$ be the tilting bundles in $\cd^b(X)$ and $\cd^b(Y)$ respectively which induce $\cd^b(X)\simeq \cd^b(A)$ and $\cd^b(Y)\simeq \cd^b(B)$. Then
$\Ext^i_X(E,E)=0$ and $\Ext^i_Y(F,F)=0$ for all $i>0$, and $\Hom_X(E,E)\cong A$, $\Hom_Y(F,F)\cong B$. By Lemma 3.4.1 in \cite{BoVdB03}, we get that $E\boxtimes F$ generates $\cd^b(X\times Y)$.
Applying the K\"{u}nneth formula for locally free sheaves on algebraic varieties, we get
\[\begin{array}{ccc}\Hom_{X\times Y}(E\boxtimes F,E\boxtimes F)&\cong& \H^0(X\times Y,({E}^\vee\otimes E)\boxtimes({F}^\vee\otimes F))\\
&\cong & \H^0(X,{E}^\vee\otimes E)\otimes\H^0(Y,{F}^\vee\otimes F)\,\,\\
&\cong& \Hom_X(E,E)\otimes \Hom_Y(F,F).\quad\quad \\
&\cong& A\otimes B.\hspace{4cm}\end{array}\]
Since $A$ and $B$ has finite global dimensions, so is $A\otimes B$.

Also applying the K\"{u}nneth formula for locally free sheaves on algebraic varieties, for any $i>0$, we get
\[\begin{array}{ccc}\Ext^i_{X\times Y}(E\boxtimes F,E\boxtimes F)&\cong& H^i(X\times Y, (E^\vee\otimes E)\boxtimes (F^\vee\otimes F))\hspace{5.8cm}\\
&\cong& \bigoplus_{j+k=i}H^j(X,(E^\vee\otimes E) )\otimes H^k(Y,(F^\vee\otimes F) )\hspace{3.9cm}\\
&\cong& \bigoplus_{j+k=i}\Ext^j_X(E,E)\otimes \Ext^k_Y(F,F)\hspace{5.7cm}\\
&=&0.\hspace{10.6cm}\end{array}\]
The last equality holds since for any $i>0$, either $j>0$ or $k>0$.

Therefore, $E\boxtimes F$ is a titling bundle in $\cd^b(X\times Y)$ and then $\cd^b(X\times Y)\simeq \cd^b(A\otimes B)$.
\end{proof}

\begin{example}
From \cite{Beilinson79}, we know that $\cd^b(\P^1)$ is triangulated equivalent to $\cd^b(A)$, where $A=k(\vec{Q})$ and $\vec{Q}:\xymatrix{\circ \ar@<-0.5ex>[r]_{a_2} \ar@<0.5ex>[r]^{a_1}& \circ}$. So $\cd^b(\P^1\times \P^1)$ is triangulated equivalent to $\cd^b(A\otimes A)$, where $A\otimes A\cong k\vec{Q}_1/I$ and $\vec{Q}_1$ is the following quiver and $I=\langle a_ic_j -b_jd_i |i,j=1,2  \rangle$.
\[\xymatrix{\circ   \ar@<-0.5ex>[rr]_{c_2} \ar@<0.5ex>[rr]^{c_1}&& \circ   \\
\\
\circ  \ar@<-0.5ex>[uu]_{a_2} \ar@<0.5ex>[uu]^{a_1} \ar@<-0.5ex>[rr]_{b_1} \ar@<0.5ex>[rr]^{b_2} &&\circ  \ar@<-0.5ex>[uu]_{d_1} \ar@<0.5ex>[uu]^{d_2}}\]
\end{example}

\end{document}